\documentclass[a4paper,11pt]{article}
\usepackage{amsmath,amsthm,amssymb,dsfont,graphicx,xspace,epsfig}
\usepackage[dvipsnames]{xcolor}
\usepackage[T1]{fontenc}
\usepackage[plain]{fullpage}
\usepackage{thm-restate}
\usepackage[hidelinks]{hyperref}
\usepackage{bm}
\usepackage{color}
\usepackage{mathrsfs}
\usepackage{cleveref}
\usepackage[shortlabels]{enumitem}
\usepackage{float}
\usepackage{dutchcal}
\usepackage{pict2e}
\usepackage{mathtools}
\usepackage{tikz}
\usepackage{subfigure}
\usepackage{xstring,refcount}
\usepackage{authblk}

\usetikzlibrary{fit, backgrounds, matrix, arrows.meta}
\usetikzlibrary{calc,arrows,positioning}
\usetikzlibrary{decorations}
\usetikzlibrary{decorations.pathmorphing}
\usetikzlibrary{decorations.pathreplacing}
\usetikzlibrary{decorations.shapes}
\usetikzlibrary{decorations.text}
\usetikzlibrary{decorations.markings}
\usetikzlibrary{decorations.fractals}
\usetikzlibrary{decorations.footprints}

\tikzset{vertex/.style = {circle,fill=black,minimum size=6pt, inner sep=0pt, outer sep=3pt}}
\tikzset{svertex/.style = {circle,fill=black,minimum size=5pt, inner sep=0pt, outer sep=2pt}}
\tikzset{labelledvertex/.style = {circle,fill=none,draw,very thick, inner sep=2pt, outer sep=1pt}}
\tikzset{Xrect/.style = {rectangle,fill=none,draw,very thick, minimum width=1.2cm, minimum height=0.8cm, outer sep=0pt}}
\tikzset{edge/.style = {thick,dashed,->,> = latex,dash pattern={on 5pt off 4pt}}}
\tikzset{arc/.style = {thick,->,> = latex}}
\tikzset{sarc/.style = {thin,->,> = latex}}
\tikzset{digon/.style = {thick,<->,> = latex}}
\tikzset{bigvertex/.style = {shape=circle,draw}}

\definecolor{g-blue}{rgb}{0.0, 0.5, 1.0}
\definecolor{g-green}{rgb}{0.4, 0.9, 0.4522}

\newtheorem{theorem}{Theorem}
\newtheorem{lemma}[theorem]{Lemma}
\newtheorem{corollary}[theorem]{Corollary}
\newtheorem{proposition}[theorem]{Proposition}
\newtheorem{conjecture}[theorem]{Conjecture}

\newtheorem{claim}{Claim}[theorem]
\newcounter{oldthm}{}
\def\@extractthmnum#1.#2{#2}

\newtheorem{remark}[theorem]{Remark}

\theoremstyle{definition}

\newenvironment{proofclaim}[1][]{\par\noindent {\it Proof of claim}. }{ \hfill$\lozenge$\par\addvspace{6pt plus 6pt}}

\newcommand{\opentriangle}{
  \raisebox{0.2pt}{\makebox[0.77778em]{
    \setlength{\unitlength}{0.6em}
    \linethickness{0.4pt}
    \begin{picture}(1,1)
    \polygon(0,0)(1,0)(1,1)
    \end{picture}
  }}
}

\newcommand{\Ad}{\mathrm{Ad}}

\newcommand{\cho}{\chi_{\ell}}
\newcommand{\dic}{\vec{\chi}}
\newcommand{\dicho}{\dic_{\ell}}

\makeatletter
\newcommand{\overleftrightsmallarrow}{\mathpalette{\overarrowsmall@\leftrightarrowfill@}}
\newcommand{\overrightsmallarrow}{\mathpalette{\overarrowsmall@\rightarrowfill@}}
\newcommand{\overleftsmallarrow}{\mathpalette{\overarrowsmall@\leftarrowfill@}}
\newcommand{\overarrowsmall@}[3]{%
  \vbox{%
    \ialign{%
      ##\crcr
      #1{\smaller@style{#2}}\crcr
      \noalign{\nointerlineskip}%
      $\m@th\hfil#2#3\hfil$\crcr
    }%
  }%
}
\def\smaller@style#1{%
  \ifx#1\displaystyle\scriptstyle\else
    \ifx#1\textstyle\scriptstyle\else
      \scriptscriptstyle
    \fi
  \fi
}
\makeatother

\newcommand{\EE}{\mathbb{E}}
\newcommand{\PP}{\mathbb{P}}
\newcommand{\e}{\mathrm{e}}

\newcommand{\Bscr}{\mathscr{B}}

\newcommand{\Rscr}{\mathscr{R}}

\renewcommand{\epsilon}{\varepsilon}

\renewcommand{\phi}{\varphi}

\let\leq\leqslant
\let\geq\geqslant

\hypersetup{
    colorlinks,
    linkcolor={red!50!black},
    citecolor={blue!50!black},
    urlcolor={blue!80!black}
}

\title{On the list version of a conjecture of Erd\H{o}s and Neumann-Lara}

\author[,1]{Ararat Harutyunyan\thanks{Research supported by ANR-21-CE48-0012.}}
\author[,2]{Lucas Picasarri-Arrieta\thanks{Research supported by JST ASPIRE JPMJAP2302.}}
\author[,1]{Gil Puig i Surroca\protect\footnotemark[1]$^,$\thanks{Research supported by MICINN PID2022-137283NB-C22.}}

\affil[1]{LAMSADE, Université Paris Dauphine - PSL, Paris, France}
\affil[2]{National Institute of Informatics, Tokyo, Japan}
\date{}

\begin{document}

\maketitle

\begin{abstract}

The \emph{dichromatic number} of a digraph $D$, denoted by $\vec{\chi}(D)$, is the smallest number of colours required to colour the vertices of $D$ such that each colour class induces an acyclic digraph. A conjecture of Erd\H{o}s and Neumann-Lara \cite{erdosPNCN1979} states that there exists a function $f(k)$ such that for every graph $G$ with $\chi(G) \geq f(k)$ there is an orientation of $G$ such that the resulting digraph $D$ satisfies $\vec{\chi}(D) \geq k$. We prove the list version of this conjecture: if $G$ has large list chromatic number then there is an orientation of $G$ such that the resulting digraph has large list dichromatic number. The main tool in our result is the following theorem, which is an extension of an analogous result of Alon \cite{alonRSA16} for the chromatic number: every graph of minimum degree $d$ admits an orientation such that the resulting digraph has list dichromatic number of order at least $\ln d$.

\medskip

\noindent{}{\bf Keywords:} list chromatic number, list dichromatic number, minimum degree, Erd\H{o}s--Neumann-Lara conjecture.
\end{abstract}

\section{Introduction }

A \emph{dicolouring} of a digraph $D$ is a function $\varphi\colon V(D)\rightarrow\mathbb Z^+$ such that, for each $i\in\mathbb Z^+$, the subgraph of $D$ induced by $\varphi^{-1}(i)$ has no directed cycles, and the \emph{dichromatic number} $\vec\chi(D)$ of $D$ is the minimum of $|\varphi(V(D))|$ over all dicolourings $\varphi$ of $D$. This notion, introduced by Erd\H os and Neumann-Lara~\cite{erdosPNCN1979,neumannlaraJCT33}, generalises the chromatic number of graphs in the following sense: if $D$ is the digraph obtained from a graph $G$ by replacing its edges by pairs of oppositely oriented arcs, then $\vec\chi(D)=\chi(G)$.

An \emph{orientation} of a graph $G$ is a digraph with underlying graph $G$ that has at most one arc between each pair of vertices. Let $\vec\chi(G)$ denote the maximum of $\vec\chi(D)$ over all orientations~$D$ of~$G$. Note that we always have $\vec\chi(G)\leq\chi(G)$. A famous open problem of Erd\H os and Neumann-Lara suggests that, conversely, $\chi(G)$ can be bounded from above in terms of $\vec\chi(G)$. 

\begin{conjecture}[\cite{erdosPNCN1979}]\label{conj:ENL} 
For every integer $k$ there is an integer $f(k)$ such that, for every graph $G$, $\chi(G)\geq f(k)$ implies $\vec\chi(G)\geq k$.
\end{conjecture}

For instance, one can take $f(1)=1$ and $f(2)=3$, but it is already unclear whether $f(3)$ exists or not. Only few results are known regarding this conjecture. Notably, Mohar and Wu~\cite{moharFMS4} were able to prove its fractional version; a similar statement for the inverse of the acyclicity ratio holds as well~\cite{puigTHESIS}. In the present paper, we address and settle the list version of the conjecture.

\begin{conjecture}[\cite{puigTHESIS}] \label{conj:ENL_list}
    For every integer $k$ there is an integer $f(k)$ such that, for every graph $G$, $\chi_{\ell}(G)\geq f(k)$ implies $\vec\chi_{\ell}(G)\geq k$.
\end{conjecture}

List colouring was introduced independently by Vizing~\cite{vizingDA29} and Erd\H{o}s, Rubin, and Taylor~\cite{erdosCN26}.
A {\it $k$-list assignment} $L$ of a graph $G$ is a function associating to each vertex $v$ of $G$ a finite set $L(v)$ of at least $k$ positive integers. An {\it $L$-colouring} is a colouring $\varphi$ of $G$ (a function $\phi\colon V(G) \to \mathbb{\mathbb Z^+}$) such that $\phi(v)\in L(v)$ for every vertex $v\in V(G)$.
A colouring $\phi$ of $G$ is {\it proper} if adjacent vertices receive distinct colours through $\phi$, and the {\it choosability} $\cho(G)$ of $G$ is the least integer $k$ such that, for every $k$-list assignment $L$, $G$ admits a proper $L$-colouring.

Clearly, every graph $G$ satisfies $\cho(G)\geq \chi(G)$. However, $\cho(G)$ cannot be bounded in terms of $\chi(G)$: Erd\H{o}s, Rubin, and Taylor~\cite{erdosCN26} showed that the complete bipartite graph $K_{n,n}$ has choosability $\Omega(\ln n)$.
In 2000, Alon~\cite{alonRSA16}, qualitatively improving on this result, proved that every graph $G$ with minimum degree $d$ has choosability $\Omega(\ln d)$, which is essentially best possible as $K_{n,n}$ has choosability $O(\ln n)$~\cite{erdosCN26}. The asymptotically optimal constant factor was finally determined by Saxton and Thomason~\cite{saxtonIM201}.

Analogously to the undirected case, the {\it dichoosability} $\dicho(D)$ of a digraph $D$ is the least integer $k$ such that, for every $k$-list assignment $L$ of $D$, $D$ admits an $L$-dicolouring. 
For an undirected graph $G$, we denote by $\dicho(G)$ the maximum of $\dicho(D)$ over all the orientations $D$ of~$G$. This directed version of list colouring was first introduced by Bensmail, Harutyunyan, and Le~\cite{bensmailJGT87}, who proved that $\dic(K_{n,n}) = \Theta(\ln n)$, hence extending the original result of Erd\H{o}s, Rubin, and Taylor. We subsume both this result and Alon's by proving the following.

\begin{theorem}
    \label{thm:main}
    Every graph $G$ with minimum degree $d$ satisfies 
    $\displaystyle
    \dicho(G) \geq (\tfrac{1}{3}-o(1))\log_2 d$.
\end{theorem}

Again, the bound is essentially tight, as $\dicho(K_{n,n})\leq \log_2 n+2$, see~\cite[Theorem~3.2]{bensmailJGT87}.
Our proof actually shows that almost all orientations $D$ of $G$ satisfy $\dicho(D) = \Omega(\ln d)$. As a direct corollary, we confirm Conjecture~\ref{conj:ENL_list}.

\begin{corollary}
    For every graph $G$, $\dicho(G) \geq (\tfrac{1}{3}-o(1))\log_2\cho(G)$.
\end{corollary}
\begin{proof}
Assume that $\chi_{\ell}(G)=k$. Then $G$ has a subgraph $H$ with minimum degree at least $k-1$. If not, let $G'$ be a minimal subgraph of $G$ with $\chi_{\ell}(G')=k$ and let $v$ be a vertex of $G'$ of degree at most $k-2$. Then, for every $(k-1)$-list assignment $L$ of $G'$, a proper $L$-colouring of $G'$ can be found by extending a proper $L|_{V(G')\setminus\{v\}}$-colouring of $G'-v$ to $v$. Thus $\chi_{\ell}(G')\leq k-1$, a contradiction.  By Theorem~\ref{thm:main}, $\vec\chi_{\ell}(G)\geq\vec\chi_{\ell}(H)\geq(\tfrac{1}{3}-o(1))\log_2(k-1)=(\tfrac{1}{3}-o(1))\log_2 \chi_{\ell}(G)$.
\end{proof}

Determining the best possible $f$ for Conjecture~\ref{conj:ENL_list} remains open. For all we know, even $\dicho(G)=\Omega\left(\cho(G)/\ln\cho(G)\right)$ could hold in general, similarly to what happens with the fractional dichromatic number and the inverse of the acyclicity ratio~\cite{moharFMS4,puigTHESIS}. This is linked to an older, analogous problem concerning improper colourings~\cite{Kang13}.

In view of Theorem~\ref{thm:main}, a natural question is whether every digraph $D$ with high minimum in-degree $\delta^-(D)$ and minimum out-degree $\delta^+(D)$ has large dichoosability. It turns out that this is not the case, as shown by the following construction due to Rambaud {\it et al.}\footnote{The construction was obtained by Rambaud, during discussions with Aboulker, Havet, Lochet, Lopes and Picasarri-Arrieta.}~\cite{perso_rambaud}. 
A \emph{tournament} is an orientation of a complete graph.

\begin{theorem}[Rambaud {\it et al.}~\cite{perso_rambaud}]
    For every positive integer $d$, there exists a tournament $T$ with $\min(\delta^-(T), \delta^+(T)) \geq d$ and $\dicho(T) \leq 2$.
\end{theorem}
\begin{proof}
    Let $T$ be the tournament defined as follows. We let $T$ contain $n=2d(d+1)$ vertices labelled $u_1,\dots,u_n$, that we partition into $r=2(d+1)$ sets $U_1,\dots,U_{r}$ of size $d$, where $U_i = \{u_{1+d(i-1)}, \ldots, u_{di}\}$ for every $i\in [1,r]$.
    We add to $T$ the arcs of the sets
    \[
        \{ vu_i : i\in [1,d] \mbox{~and~} v\in U_{1+d+i}\}
    \mbox{~~~and~~~} 
        \{u_{n+1-i}v : i\in [1,d] \mbox{~and~} v\in U_{1+i}\}.
    \]
    For every remaining pair of distinct non-adjacent vertices $u_i$ and $u_j$, we add the arc $u_iu_j$ if $i<j$ and the arc $u_ju_i$ otherwise.
    It is straightforward to check that every vertex in $T$ has in- and out-degree at least~$d$.

    Note that, for every digraph $D$ and ordering $\prec$ on $V(D)$, any proper colouring of the undirected graph $D^\prec$ with vertex set $V(D)$ containing an edge $uv$ whenever $v\prec u$ and $uv \in A(D)$ (called the \emph{backedge graph} of $D$ with respect to $\prec$) yields a dicolouring of $D$. To see this, it is sufficient to notice that every independent set of $D^\prec$ induces an acyclic digraph on $D$.

    By construction, the backedge graph $G$ of $T$ with respect to the ordering $u_1,\dots, u_n$ is a disjoint union of stars, and in particular $\dicho(T) \leq \cho(G) \leq 2$.
\end{proof}

Relating to what upper bounds could be a counterpart to Theorem~\ref{thm:main}, a remarkable conjecture of Alon and Krivelevich~\cite{alonAC2} asserts that the choice number of bipartite graphs of maximum degree $d$ is at most $O(\ln d)$. The best current bound, due to Bradshaw, Mohar, and Stacho~\cite{bradshawArXiv24}, is $(\frac{4}{5}+o(1))\frac{d}{\ln d}$; in parallel, the asymmetric version of the conjecture has also been treated~\cite{alonkang}. We do not know if the analogous question for oriented bipartite graphs is easier. 

\section{Preliminaries}

Our proof of Theorem~\ref{thm:main} combines probabilistic arguments with a trick due to Kühn and Osthus~\cite{kuhnComb24}. Similarly to the proof of Alon~\cite{alonRSA16}, we build our list assignment $L$ randomly. In the original proof of Alon, it is shown that, with positive probability, at least one edge of $G$ is monochromatic in every $L$-colouring. We push the analysis further and actually obtain that, in every $L$-colouring, a large subgraph of $G$ is monochromatic. 
It is large enough so that, with positive probability, in a random orientation of $G$, none of these possible monochromatic large subgraphs is acyclic.

The remaining of this section is a collection of well-known results used later on.

\begin{lemma}[{\sc Chernoff}]
    \label{lem:chernoff}
    If $X$ is a random variable following a binomial law with parameters $p \in [0, 1]$ and $n \geq 0$, with expectation $\EE(X) = \mu = np$, then
    \[
        \PP (X\leq (1-\epsilon)\mu) \leq \exp(-\tfrac{\epsilon^2}{2}\mu)
    \]
    for any $0<\epsilon < 1$.
\end{lemma}

\begin{lemma}[\cite{manberSJC13}]
    \label{lem:prob_orientation_acyclic}
    Let $G$ be a graph with average degree $\Gamma$ and order $n$, and let $D$ be a random orientation of $G$ with the uniform distribution. Then 
    \[
        \PP(D \mbox{ is acyclic}) \leq \left(\frac{\Gamma+1}{2^{\frac{\Gamma}{2}}} \right)^{n}.
    \]
\end{lemma}

The following result is due to Kühn and Osthus~\cite{kuhnComb24}, see~\cite{charArXiv25} for a proof of the exact statement.

\begin{lemma}[\cite{charArXiv25}]
    \label{lemma:KO1}
    Let $\Gamma,d$ be real numbers with $\Gamma> 16d\geq 32$. Every bipartite graph $G$ with average degree $\Ad(G) =\Gamma$ contains an induced subgraph $G^\star$ with bipartition $(A^\star, B^\star)$ such that $|A^\star| \geq \frac{\Gamma}{128d}|B^\star|$ and, for every $a\in A^\star$, $4d \leq d_{G^\star}(a) \leq 64d$.
\end{lemma}

We make use of the following straightforward calculations several times.

\begin{proposition}
    \label{prop:simplify}
    For every integer $r$ large enough,
    \[
        \binom{r}{\lfloor r/2\rfloor} \leq 2^r
        \mbox{~~~~~and~~~~~}
        \binom{\lfloor r^2/2\rfloor}{r}\cdot \binom{r^2}{r}^{-1} \geq \frac{1}{2^{r+2}}.
    \]
\end{proposition}
\begin{proof}
    The first inequality follows from the fact that $\sum_{i=0}^r \binom{r}{i}=2^r$. For the second inequality, we have that
    \[
        \frac{\binom{\lfloor r^2/2\rfloor}{r}}{ \binom{r^2}{r}}
        =\frac{\lfloor r^2/2\rfloor(\lfloor r^2/2\rfloor-1)\cdots(\lfloor r^2/2\rfloor-r+1)}{r^2(r^2-1)\cdots(r^2-r+1)}
        \geq\frac{1}{2^r}\left(\frac{r^2-2r+1}{r^2-r+1}\right)^r
        \geq\frac{1}{2^{r+2}},
    \]
    where in the last inequality we used that $\left(\frac{r^2-2r+1}{r^2-r+1}\right)^r=\left(1-\frac{r}{r^2-r+1}\right)^r\rightarrow \e^{-1}$ as $r\rightarrow\infty$ and that $r$ is large enough.
\end{proof}

\section{The proof}

\begin{lemma}
    \label{lemma:key_lem}
    There exist $r_0,k_0\in \mathbb{N}$ such that the following holds for all integers $r\geq r_0$ and $k\geq k_0$. Let $G=(A\cup B, E)$ be a bipartite graph such that:
    \begin{enumerate}[label=(\roman*)]
        \item every vertex $a\in A$ has degree $d(a)\geq k (\ln k)^2\cdot r^32^{r+2}$, and
        \item $|A|\geq r2^{r+3} |B|$.
    \end{enumerate}
    Then, there exists an $r$-uniform list assignment $L$ such that, for every $L$-colouring $\gamma$ of $G$, $G$ admits a subgraph $H_{\gamma}$ such that:
    \begin{itemize}
        \item $|H_\gamma| \geq \frac{1}{2^{r+4}}|G|$,
        \item $H_{\gamma}$ has average degree at least $k$, and
        \item for every edge $uv$ of $H_\gamma$, $\gamma(u) = \gamma(v)$.
    \end{itemize}
\end{lemma}

\begin{proof}
    We do not give the exact values of $r_0$ and $k_0$, we simply assume they are large enough so that all upcoming inequalities hold.
    We let $R=\{1,\dots,r^2\}$, and denote by $\Rscr = \binom{R}{\lfloor r^2/2\rfloor}$ the family of subsets of $R$ of size precisely $\lfloor r^2/2\rfloor$. 
    
    We assign to each vertex $b\in B$ an element $L(b)$ from $\binom{R}{r}$ sampled uniformly at random. 
    We say that a vertex $a\in A$ is {\it saturated} if every element $P\in \Rscr$ contains the list of at least $\frac{1}{2}kr^2$ neighbours of $a$. That is, $a$ is saturated if, for every $P\in \Rscr$,
    \[
        |\{b\in N(a) : L(b)\subseteq P\}| \geq \tfrac{1}{2}kr^2.
    \]
    
    \begin{claim}
        With positive probability, at least half of the vertices in $A$ are saturated.
    \end{claim}
    \begin{proofclaim}
        For a fixed element $P\in \Rscr$ and a fixed vertex $b\in B$, observe that $L(b) \subseteq P$ with probability at least
        \[
            \binom{\lfloor r^2/2\rfloor}{r}\cdot \binom{r^2}{r}^{-1} \geq \frac{1}{2^{r+2}},
        \]
        the inequality coming from \Cref{prop:simplify}. It follows that, for any fixed vertex $a\in A$, the probability that a fixed element $P\in \Rscr$ contains the lists of at most $\ell = \lfloor \frac{1}{2}kr^2-\frac{1}{2} \rfloor$ of its neighbours is at most 
        \begin{align*}
            \binom{d(a)}{\ell} \cdot \left( 1-\frac{1}{2^{r+2}} \right)^{d(a)-\ell}
            &\leq \exp\left( \frac{1}{2}kr^2 \ln d(a) - \frac{d(a)-\tfrac{1}{2}kr^2}{2^{r+2}}\right)\\
            &\leq \exp\Bigg( \frac{1}{2}kr^2\ln\Big(k(\ln k)^2r^32^{r+2}\Big) - k(\ln k)^2r^3 + k\Bigg)\\
            &< \exp\left(-\frac{1}{2}k(\ln k)^2r^3\right),
        \end{align*}
        where in the first inequality we use that $1-x \leq \e^{-x}$ for every $x>0$, in the second inequality we use that the expression is decreasing with $d(a)$ for $d(a)\geq kr^2 2^{r+1}$, and in the last inequality we use that $r$ and $k$ are large enough.
        By the union bound and \Cref{prop:simplify}, it follows that
        \begin{align*}
            \PP(a \mbox{ is not saturated}) &< \binom{r^2}{\lfloor r^2/2\rfloor} \cdot \e^{-\frac{1}{2}k(\ln k)^2r^3} \leq 2^{r^2}  \cdot \e^{-\frac{1}{2}k(\ln k)^2r^3}< \frac{1}{2},
        \end{align*}
        where in the last inequality we use again that $r$ is large. Let $A_{\rm s}$ denote the set of saturated vertices of $A$. By linearity of expectation, 
        \[
            \EE(|A_{\rm s}|) \geq \frac{1}{2}|A|,
        \]
        and in particular 
        \[
        \PP\left( |A_{\rm s}| \geq \frac{1}{2}|A| \right) >0.
        \]
    \end{proofclaim}
    From now on, we fix an $r$-list assignment $L_B$ of $B$ for which at least half of the vertices in $A$ are saturated, the existence of which is guaranteed by the claim above. We let $A_{\rm s}$ be the set of saturated vertices of $A$.

    Now we assign to each vertex $a\in A$ an element $L(a)$ from $\binom{R}{r}$, again, sampled uniformly at random. Given any $L_B$-colouring $\beta$ of $B$, a vertex $a\in A$ is {\it truly saturated} with respect to $\beta$ if, for every $c\in L(a)$, $a$ has at least $k$ distinct neighbours $b$ such that $\beta(b) = c$.
    For every $L_B$-colouring $\beta$ of $B$, we denote by $X_{\beta}$ the number of truly saturated vertices in $A$ with respect to $\beta$.

    \begin{claim}
        With positive probability, $X_{\beta} \geq \frac{1}{2^{r+4}}|A|$ for every $L_B$-colouring $\beta$ of $B$.
    \end{claim}
    \begin{proofclaim}
        Let us fix any $L_B$-colouring $\beta$ of $B$. 
        We let $A^\star$ be the set of vertices $a\in A_{\rm s}$ that are truly saturated (with respect to $\beta$). In particular, $X_{\beta}\geq |A^\star|$.
        We first argue that it is very likely that $|A^\star|$ is large.
        For this, let us fix a saturated vertex $a\in A_{\rm s}$. A colour $c\in \{1,\dots,r^2\}$ is {\it available} for $a$ if $a$ has at most $k-1$ neighbours coloured $c$ via $\beta$.
        Note that $a\in A^\star$ if and only if $L(a)$ is disjoint from the set of colours available for $a$.
        
        We claim that there are at most $\frac{1}{2}r^2-1$ available colours for $a$. To see this, assume that $\lfloor r^2/2\rfloor$ colours are available for $a$, and let $P$ be be any set of exactly $\lfloor r^2/2\rfloor$ such colours. By definition of being saturated, $a$ has at least $\frac{1}{2}kr^2$ neighbours $b$ such that $L_B(b) \subseteq P$. By the Pigeonhole Principle, among these neighbours, $k$ share a common colour $c\in P$ via $\beta$, hence implying that $c$ is not available for $a$, a contradiction.

        Therefore, among the elements of $\binom{R}{r}$, at least $\binom{\lceil r^2/2\rceil}{r}$ are disjoint from the set of colours available for $a$. By \Cref{prop:simplify}, it follows that 
        \begin{equation*}
            \PP(a \in A^\star) \geq \binom{\lceil r^2/2\rceil}{r}\cdot \binom{r^2}{r}^{-1} \geq \frac{1}{2^{r+2}}.
        \end{equation*}
        Note that the events $a\in A^\star$ and $a'\in A^\star$ are independent for every $a'\in A_{\rm s}\setminus\{a\}$.
        In particular, for every integer $\ell$, we have
        \[
            \PP\Big(|A^\star| \leq \ell \Big) \leq \PP\Big( \Bscr \leq \ell \Big),
        \]
        where $\Bscr$ is a random variable following a binomial law with parameters $p = \frac{1}{2^{r+2}}$ and $n = |A_{\rm s}|$. By Chernoff's inequality (\Cref{lem:chernoff}), we thus have
        \begin{align*}
            \PP\left(|A^\star| \leq \frac{1}{2^{r+3}}|A_{\rm s}| \right) &\leq \PP\left( \Bscr \leq \frac{1}{2^{r+3}}|A_{\rm s}| \right),\\
            &\leq \PP\left( \Bscr \leq \tfrac{1}{2}\cdot \EE(\Bscr) \right)\\
            &\leq \exp\left( -\frac{1}{8}\cdot \frac{1}{2^{r+2}} |A_{\rm s}|  \right)\\
            &\leq \exp\left( -\frac{r}{8} |B| \right),
        \end{align*}
        where in the last inequality we used that $|A_{\rm s}| \geq \frac{1}{2}|A| \geq r2^{r+2}|B|$. 
        Recall that all vertices in $A^\star$ are truly saturated, so $X_{\beta} \geq |A^\star|$, and it follows from the inequality above that 
        \[
            \PP\left( X_{\beta} \leq \frac{1}{2^{r+4}}|A| \right) \leq \PP\left( X_{\beta} \leq \frac{1}{2^{r+3}}|A_{\rm s}| \right) \leq \exp\left( -\frac{r}{8} |B| \right).
        \]
        There are $r^{|B|}$ distinct $L_B$-colourings of $B$. By the union bound, it thus follows that
        \[
            \PP\left(\exists \beta: X_{\beta}\leq \frac{1}{2^{r+4}}|A|\right)  \leq 
            \exp\Big( (-\tfrac{1}{8}r + \ln r)|B|\Big) < 1,
        \]
        where, in the last inequality, we use that $r$ is sufficiently large.
    \end{proofclaim}

    From now on, we fix a list assignment $L_A$ such that $|X_{\beta}| \geq \frac{1}{2^{r+4}}|A|$ for every $L_B$-colouring $\beta$ of $B$, the existence of which is guaranteed by the claim above.

    We let $L= L_A\cup L_B$, and claim that the statement holds for $L$. To see this, let $\gamma$ be any $L$-colouring of $G$ and let $\beta$ be its restriction to $B$. Let $A'$ be the set of vertices in $A$ that have at least $k$ neighbours in its own colour class. By choice of $L$, we have $|A'| \geq \frac{1}{2^{r+4}}|A| \geq |B|$. Let $H_\gamma$ be the bipartite graph with vertex set $A'\cup B$ containing all monochromatic edges of $G[A'\cup B]$ (that is, edges $uv$ of $G[A'\cup B]$ such that $\gamma(u) = \gamma(v)$). The result follows.
\end{proof}

\begin{remark}
It seems that a version of Lemma \ref{lemma:key_lem} can also be obtained using the proof method of Theorem 6 in \cite{Kang13}. As we did not try to optimize the constants in the statement of the lemma, it is not completely clear which of these approaches would yield the best constants.
\end{remark}

\begin{lemma}
    \label{lemma:key_lem_dic}
    There exists $r_0\in \mathbb{N}$ such that the following holds for every integer $r\geq r_0$. Let $G=(A\cup B, E)$ be a bipartite graph such that:
    \begin{enumerate}[label=(\roman*)]
        \item every vertex $a\in A$ has degree $d(a)\geq r^6 2^{2r+8}$, and
        \item $|A|\geq r2^{r+3}|B|$.
    \end{enumerate}
    Then, there exists an orientation $D$ of $G$ such that $\dic_\ell(D) \geq r$.
\end{lemma}
\begin{proof}
    Let $n$ denote the order of $G$.
    As $r$ is large enough, by Lemma~\ref{lemma:key_lem} applied with $k=r2^{r+6}$, there exists an $r$-list assignment $L$ such that, for every $L$-colouring $\gamma$ of $G$, $G$ admits a subgraph $H_\gamma$ with at least $\frac{n}{2^{r+4}}$ vertices and average degree at least $r2^{r+6}$, whose edges are all monochromatic. We fix such a list assignment and such a graph $H_\gamma$ for each $L$-colouring $\gamma$. 
    
    To prove the result, it is sufficient to justify the existence of an orientation $D$ of $G$ for which, for every $L$-colouring $\gamma$ of $G$, the corresponding induced orientation $D_{\gamma}$ of $H_\gamma$ contains a directed cycle. Indeed, any cycle of $H_{\gamma}$ belongs to a connected component of $H_{\gamma}$, and hence it is monochromatic under $\gamma$. Let us check that, if we let $D$ be a random orientation of $G$, the above property holds with positive probability.
    Using  \Cref{lem:prob_orientation_acyclic} and the fact that $r$ is large enough, we have that
    \[
        \PP(D_{\gamma} \mbox{ is acyclic}) \leq 2^{(\log_2(r2^{r+6}+1)-r2^{r+5})|H_{\gamma}|} \leq 2^{-r2^{r+4}\cdot|H_{\gamma}|} \leq 2^{-rn}
    \]
    for every $\gamma$. Recall that there exist at most $r^n$ $L$-colourings of $G$. Therefore, by the union bound,
    \[
        \PP(\exists \gamma : D_{\gamma} \mbox{ is acyclic}) \leq 2^{(-r + \log_2 r)n} < 1.
    \]
    The result follows.
\end{proof}

Now Theorem~\ref{thm:main} follows directly from the following.

\begin{restatable}{theorem}{mainthm}
    There exists $r_0\in \mathbb{N}$ such that the following holds for every integer $r\geq r_0$. 
    Every graph $G$ with average degree $\Ad(G) \geq r^72^{3r+17}$ satisfies $\dicho(G)\geq r$. 
\end{restatable}
\begin{proof}
    Let us fix such a graph $G$. By taking a maximum cut of $G$, it is straightforward to check that $G$ contains, as a subgraph, a bipartite graph with average degree $\Gamma \geq \frac{1}{2}\Ad(G) \geq r^72^{3r+16}$. By Lemma~\ref{lemma:KO1} applied with $d=r^6 2^{2r+6}$, it follows that $G$ contains a bipartite graph $G'$ with bipartition $(A',B')$ such that $|A'| \geq r2^{r+3}|B'|$ and every vertex in $A'$ has degree at least $r^6 2^{2r+8}$.
    The result then follows from Lemma~\ref{lemma:key_lem_dic}.
\end{proof}

\section*{Acknowledgements} We are thankful to Ross Kang for helpful suggestions.

\bibliographystyle{abbrv}
\bibliography{refs}

\end{document}